\documentclass[11pt,twoside]{article}
\usepackage{amssymb}
\usepackage{amssymb,amsmath,amsthm,amsfonts,mathrsfs,hyperref}
\usepackage{times}
\usepackage{enumerate}
\usepackage{cite,titletoc}

\allowdisplaybreaks

\def\N{\mathop{\mathbb N\kern 0pt}\nolimits}
\def\Q{\mathop{\mathbb Q\kern 0pt}\nolimits}
\def\R{\mathop{\mathbb R\kern 0pt}\nolimits}
\def\SS{\mathop{\mathbb S\kern 0pt}\nolimits}

\hypersetup{colorlinks=true,linkcolor=blue,citecolor=red,urlcolor=cyan}

\theoremstyle{plain}
\newtheorem{theorem}{Theorem}[section]

\newtheorem{lemma}[theorem]{Lemma}

\theoremstyle{definition}

\newtheorem{remark}{Remark}[section]

\numberwithin{equation}{section}

\title{The factorizations of $H^\rho(\mathbb{R}^n)$ via multilinear Calder\'{o}n-Zygmund operators on weighted Lebesgue spaces
\footnote{This project is partially supported by the Natural Science Foundation of Anhui Province (No.2108085QA19)
and the National Natural Science Foundation of China(Nos.12101010,12071223,11971237,11771223)}\\}

\author{Dinghuai Wang, Rongxiang Zhu and Lisheng Shu\\
[12pt] {\small School of Mathematics and Statistics, Anhui Normal University, Wuhu, 241002, China}\\
    [12pt] {\small  Email: Wangdh1990@126.com; ZRx1268@163.com; shulsh@mail.ahnu.edu.cn.}\\}

\begin{document}

\date{}
\maketitle
\thispagestyle{empty}

\begin{abstract}
We extend the recently much-studied Hardy factorization theorems to the weight case. The key point of this paper is to establish the factorization theorems without individual condition on the weight functions. As a direct application, we obtain the characterizations of $BMO(\mathbb{R}^n)$ space and Lipschitz spaces via the weighted boundedness of commutators of multilinear Calder\'{o}n-Zygmund operators with the genuinely multilinear weights.

\vskip 0.2 true cm


\vskip 0.2 true cm
\noindent
\noindent
\end{abstract}

\vskip 0.6 true cm

\section{Introduction}

\qquad  We briefly summarize some classical and recent works in the literature, which lead to the results presented here. The first result on the weak factorization theorem of Hardy space $H^{1}(\mathbb{R}^n)$ in terms of Caldr\'{o}n-Zygmund operators is due to the work of Cofiman, Rochberg and Wiess \cite{CRW1976}. The result depends upon the duality between $H^{1}(\mathbb{R}^n)$ and $BMO(\mathbb{R}^n)$ and upon a new result linking $BMO(\mathbb{R}^n)$ and the $L^{p}(\mathbb{R}^n)$ boundedness of certain commutator operators. For the multilinear case, the weak factorization of the classical Hardy space $H^{1}(\mathbb{R}^n)$ was established by Li and Wick \cite{LW2018}. On the other hand, Uchiyama in \cite{U1981} gave a proof of the weak factorization theorem of Hardy space $H^\rho(\mathbb{R}^n)$ with $\frac{n}{n+1}<\rho<1$, while the one for multilinear case is credited in \cite{K2018}. Recently, Dao and Wick \cite{DW2021} obtained the $H^{1}(\mathbb{R}^n)$ in terms of multilinear Caldr\'{o}n-Zygmund operators using Morrey spaces. In \cite{WZ2021}, we considered the results related to fractional integral operator.

Inspired by the above works, we want to provide the constructive proof of the weak factorization $H^\rho(\mathbb{R}^n)$ in terms of multilinear operators of Calder\'{o}n-Zygmund type on weighted Lebesgue spaces with the general multiple weights. Our strategy and approach will be to modify the direct constructive proof of Li and Wick \cite{LW2018} and Uchiyama in \cite{U1981}. However, in the multilinear setting the most interesting phenomena occurs when the genuinely multilinear weights are used.
Many weighted works in the multilinear setting treat each variable separately with its own Muckenhoupt class of weights, since the multi-weight condition cannot deduce that the weight function is locally integrable. For example, the initial work of multivariable Rubio de Francia extrapolation theorem was obtained by \cite{CUM} (or \cite{D2011}) for $\vec{\omega}\in A_{\vec{P}}$ with $\omega_{i}\in A_{p_{i}}$. These works treat each variable separately with its own Muckenhoupt class of weights. Recently, the long standing problem of multivariable Rubio de Francia extrapolation theorem for the multilinear Muckenhoupt classes $A_{\vec{P}}$ was showed by Li, Martell and Ombrosi in \cite{LMOarXiv}. Thus, it is very interesting to obtain the factorization theorems on weighted Lebesgue spaces using the multivariable nature of the problem.

The purpose of this article is first: to establish the factorization result for $H^{1}(\mathbb{R}^{n})$ in terms of multilinear Calder\'{o}n-Zygmund on weighted Lebesgue spaces with the general $A_{\vec{P}}$ weights; and second: to obtain the similar result for $H^{\rho}(\mathbb{R}^{n})(\frac{n}{n+\gamma}<\rho<1)$ related to $A_{\vec{P},q}$ weights. Specially, the results are new even in the linear case.

Dual to the multilinear commutator, in both language and via a formal computation, we define the multilinear "multiplication"operators $ {\textstyle \prod _{l}}$ as follows.
	\begin{align}\label{commutator}
	 {\textstyle\prod_{l}}(g,h_1,\cdots,h_m)(x):=h_lT_l^*(h_1,\cdots,h_{l-1,}g,h_{l+1},\cdots,h_m)(x)-gT(h_1,\cdots,h_m)(x),
	\end{align}
where $T_l^*$ is the $l$-th partial adjoint of $T$.

Now let us state the main results of this article. The precise definitions are given in the next section.

\begin{theorem}\label{main1}	
Let $1\le l\le m$, $1<p, p_{1},\dots,p_{m}<\infty$, $\frac{1}{p_1}+\cdots+\frac{1}{p_m}=\frac{1}{p}$ and $\vec{\omega}\in A_{\vec{P}}$. Then for any $f\in H^{1}(\mathbb{R}^{n})$, there exists sequences $  \{ \lambda_{s}^{k}    \}\in \ell_{1}$ and functions $g_{s}^{k}\in L^{p^{\prime}}(\nu^{1-p^{\prime}}_{\vec{\omega}}),h_{s,1}^{k}\in L^{p_{1}}(\omega_{1}),\dots ,h_{s,m}^{k}\in L^{p_{m}}(\omega_m)$ such that
    \begin{align}\label{f=1}
     f=\sum_{k=1}^{\infty}\sum_{s=1}^{\infty}\lambda_{s}^{k}
     {\textstyle \prod_{l}}(g_{s}^{k},h_{s,1}^{k},\dots ,h_{s,m}^{k})
     \end{align}
     in the sense of $H^{1}(\mathbb{R}^{n})$. Moreover,
      \begin{align*}
      \|f\|_{H^{1} (\mathbb{R}^{n} )}\approx
      \inf \bigg\{\sum_{k=1}^{\infty}\sum_{s=1}^{\infty} |\lambda_{s}^{k}|
      \|g_{s}^{k} \|_{L^{p^{\prime}}(\nu^{1-p^{\prime}}_{\vec{\omega}})}
      \|h_{s,1}^{k} \|_{L^{p_{1}}(\omega_1)}\cdots \|h_{s,m}^{k}\|_{L^{p_{m}}(\omega_m)}\bigg\},
      \end{align*}
where the infimum above is taken over all possible representations of $f$ that satisfy \eqref{f=1}.
\end{theorem}

\begin{theorem}\label{main2}	
Let $1\le l\le m$, $\frac{n}{n+\gamma}<\rho<1$, $1< p_{1},\dots,p_{m},q<\infty$, $\frac{1}{p_1}+\cdots+\frac{1}{p_m}+1=\frac{1}{\rho}+\frac{1}{q}$ and $\omega\in A_{\vec{P},q}$.
	Then for any $f\in H^{\rho}(\mathbb{R}^n)$, there exists sequences $  \{ \lambda_{s}^{k}  \}\in \ell_{\rho}$ and functions $ g_{s}^{k}\in L^{q'}(\mu_{\vec{\omega}}^{1-q'}),h_{s,1}^{k}\in L^{p_{1}}(\omega_1^{p_1}),\cdots ,h_{s,m}^{k}\in L^{p_{m}}(\omega_{m}^{p_m})$ such that
	\begin{align}\label{f=2}
		f=\sum_{k=1}^{\infty}\sum_{s=1}^{\infty}\lambda_{s}^{k}
		{\textstyle \prod_{l}}(g_{s}^{k},h_{s,1}^{k},\dots ,h_{s,m}^{k})
	\end{align}
	in the sense of $H^{\rho}(\mathbb{R}^n)$. Moreover,
	\begin{align*}
		\|f\|_{H^{\rho} (\mathbb{R}^n )}\approx
		\inf \bigg\{\Big(\sum_{k=1}^{\infty}\sum_{s=1}^{\infty} \big(|\lambda_{s}^{k} |
		 \|g_{s}^{k} \|_{L^{q'}(\mu_{\omega}^{1-q'})}
		 \|h_{s,1}^{k} \|_{L^{p_{1}}(\omega_1^{p_1})}\cdots \|h_{s,m}^{k} \|_{L^{p_{m}}(\omega_m^{p_m})}\big)^{\rho}\Big)^{1/\rho} \bigg\},
	\end{align*}
	where the infimum above is taken over all possible representations of $f$ that satisfy \eqref{f=2}.
\end{theorem}

Our paper is organized as follows. In the next section, we give the definitions of some functional spaces and operators the expert reader can easily skip this part. In section \ref{s3}, we provide necessary auxiliary results. Section \ref{s4} is devoted to proving Theorems \ref{main1} and \ref{main2}. In the last Section, we provide a short discussion about the applications of the weak factorization theorems.

\section{Definitions and preliminaries}\label{s2}

Let $|E|$ denote the Lebesgue measure of a measurable set $E\subset \mathbb{R}^n$, and $\omega(E)=\int_{E}\omega(x)dx$. Throughout this paper, the letter $C$ denotes constants which are independent of main variables and may change from one occurrence to another. We denote $A\lesssim B$ if $A\leq CB$ for some constant that can depend on
the dimension, Lebesgue exponents, weight constants, and on various other constants
appearing in the assumptions.

\subsection{Multilinear Caldr\'{o}n-Zygmund operators}
The multilinear Calder\'{o}n-Zygmund theory was systematically formulated by Grafakos and Torres \cite{GT2002}. Recall that $m$-Calder\'{o}n-Zygmund operator $T$ is a bounded operator which satisfies
$$\|T(f_{1},\cdots, f_{m})\|_{L^{p}(\mathbb{R}^n)}\leq C\|f_{1}\|_{L^{p_{1}}(\mathbb{R}^n)}\cdots\|f_{m}\|_{L^{p_{m}}(\mathbb{R}^n)},$$
for some $1<p_{1},\cdots,p_{m}<\infty$ with $1/p=1/p_{1}+\cdots+1/p_{m}$ and the function $K$, defined off the diagonal $y_{0}=y_{1}=\cdots=y_{m}$ in $(\mathbb{R}^{n})^{m+1}$, satisfies the conditions as follow:

(1) The function $K$ satisfies the size condition.
$$|K(y_{0},y_{1},\cdots,y_{m})|
\leq\frac{C}{\big(\sum_{k=1}^{m}|y_{k}-y_{0}|\big)^{mn}}.$$

(2) The function $K$ satisfies the regularity condition. For some $\gamma>0$ and all $1\leq i\leq m$,
if $|y_{i}-y'_{i}|\leq \frac{1}{2}\max_{0\leq k\leq m}|y_{0}-y_{k}|$,
$$|K(y_{0},\cdots,y_{i},\cdots,y_{m})-K(y_{0},\cdots,y'_{i},\cdots,y_{m})|
\leq\frac{C|y_{i}-y'_{i}|^{\gamma}}{\big(\sum_{k=1}^{m}|y_{k}-y_{0}|\big)^{mn+\gamma}}.$$
Then we say $K$ is a $m$-linear Calder\'{o}n-Zygmund kernel. If $x\notin \bigcap_{i=1}^{m}{\rm supp} f_{i} $, then
$$T(f_{1},\cdots,f_{m})(x)=\int_{\mathbb{R}^{mn}}K(x,y_{1},\cdots,y_{m})f_{1}(y_{1})\cdots f_{m}(y_{m})dy_{1}\cdots dy_{m}.$$
The $l$-th partial adjoint of $T$ is
$$T^{*}_{l}(f_{1},\cdots,f_{m})(x)=\int_{\mathbb{R}^{mn}}K(y_{l},y_{1},\cdots,y_{l-1},x,y_{l},\cdots,y_{m})f_{1}(y_{1})\cdots f_{m}(y_{m})dy_{1}\cdots dy_{m}.$$

Now we define the $mn$-homogeneous Calder\'{o}n-Zygmund operators. We say that $T$ is $mn$-homogeneous if the kernel function $K$ satisfies
$$|K(x_{0},\cdots,x_{m})| \ge \frac{C}{N^{mn}}$$
 for $m+1$ balls $B_0=B_0(x_0,r),\cdots,B_m=B_m(x_m,r)$ satisfying $ |y_0-y_l |\approx{Nr} $ for $l=1,2,\cdots,m$ and for all $x\in B_0$, where $r>0$ and $N$ a large number.

The definitions of the general commutator and the linear commutator were given by P\'{e}rez and Torres in \cite{PT2003}, which coincides with the linear commutator $[b,T]$ when $m=1$. Suppose $T$ is a $m$-linear operator and $\vec{b}=(b_{1},\cdots,b_{m})$. Define the general $l$-th commutator of $T$ with $b_{l}$ by
$$[b_{l},T]_{l}(f_{1},\cdots,f_{m})(x):=b_{l}(x)T(f_{1},\cdots,f_{i},\cdots,f_{m})(x)-T(f_{1},\cdots,b_{l}f_{l},\cdots,f_{m})(x).$$

\subsection{Muckenhoupt weights.}

For $1<p<\infty,$ recall that the Muckenhoupt $A_{p}$ \cite{M1972} of weights consists of all non-negative, locally integrable, functions $\omega$ such that
$$\sup_{Q}\bigg(\frac{1}{|Q|}\int_{Q}\omega(x)dx\bigg)\bigg(\frac{1}{|Q|}\int_{Q}\omega(x)^{-\frac{1}{p-1}}dx\bigg)^{p-1}<\infty.$$
And a weight function $\omega$ belongs to the class $A_{1}$ if
$$\frac{1}{|Q|}\int_{Q}\omega(x)dx\Big(\mathop\mathrm{ess~sup}_{x\in Q}\omega(x)^{-1}\Big)<\infty.$$
While $A_{\infty}=\bigcup_{1<p<\infty}A_{p}$. For $\omega\in A_{\infty}$, there exists $0<\epsilon,L<\infty$ such that for all balls $B$ and all measurable subsets $E$ of $B$, we have
\begin{align}\label{weight-eq1}
  \frac{\omega(E)}{\omega(B)}\leq C\Big(\frac{|E|}{|B|}\Big)^{\epsilon} \qquad \text{and} \qquad  \Big( \frac{|E|}{|B|} \Big)^{L}\le C\frac{\omega(E)}{\omega(B)}.
\end{align}

The definition of the multiple or vector weights used in the multilinear setting was given by Lerner et.al. \cite{LOPTT2009}. For $1<p_1,\dots,p_m<\infty$, $\vec{P}=(p_1,p_2,\dots,p_m)$, and
	$p$ such that $\frac{1}{p_1}+\dots +\frac{1}{p_m}=\frac{1}{p}$, a vector weight $ \vec{\omega} =(\omega_1,\omega_2,\dots, \omega_m)$ belongs to $A_{\vec{P}}$ if
	\begin{align*}
		\underset{Q}{\sup} \Big(\frac{1}{|Q|}\int_{Q}\prod_{i=1}^{m}\omega_i(x)^{\frac{p}{p_i}}dx\Big)^{\frac{1}{p}}
     \prod_{i=1}^{m} \Big(\frac{1}{|Q|}\int_{Q}\omega_{i}(x)^{1-p_i^\prime} dx\Big)^{\frac{1}{p_i^{\prime}}}< \infty.
	\end{align*}
	For brevity, we will often use the notation $\nu_{\vec{\omega}}=\prod_{i=1}^{m} \omega_{i}^{\frac{p}{p_i}}$. Lerner et al. in \cite[Theorem 3.6]{LOPTT2009} showed that the multilinear $A_{\vec{P}}$ condition has the following interesting characterization in terms of the linear $A_{p}$ classes. The weight function $\vec{\omega}\in A_{\vec{P}}$ if and only if
\begin{equation*}
     \Bigg\{
   \begin{array}{ll}\vspace{1ex}
\omega_{i}^{1-p'_{i}}\in A_{mp'_{i}}, i=1,\cdots,m,\vspace{1ex}\\
\nu_{\vec{\omega}}\in A_{mp},
   \end{array}
\end{equation*}
where the condition $\omega_{i}^{1-p'_{i}}\in A_{mp'_{i}}$ in the case $p_{i}=1$ is understood as $\omega_{i}^{1/m}\in A_{1}$.

Let $1<p_1,\cdots,p_m<\infty$, $\vec{P}=(p_1,\cdots,p_m)$, $1/m<p\leq q<\infty$ and $\frac{1}{p}=\frac{1}{p_1}+\cdots+\frac{1}{p_m}$. A vector weight $\vec{\omega}=(\omega_1,\omega_2,\cdots,\omega_m)$ belongs to $A_{\vec{P},q}$ if $$[\vec{\omega}]_{A_{\vec{P},q}}=\sup_{Q}\Big(\frac{1}{|Q|}\int_Q\prod_{i=1}^{m}\omega_i(x)^qdx\Big)^{\frac{1}{q}}
\prod_{i=1}^{m}\Big(\frac{1}{|Q|}\int_Q\omega_i(x)^{-p_i^{\prime}}dx\Big)^{\frac{q}{p_i^{\prime}}}<\infty.$$
To avoid ambiguities in the notation we will use $\nu_{\vec{\omega}}$ when dealing with $A_{\vec{P}}$ classes and $\mu_{\vec{\omega}}$ with $A_{\vec{P},q}$ ones.
It was shown by Moen in \cite{M2009} that if $\vec{\omega}\in A_{\vec{P},q}$, then $\omega_{i}^{-p_i^{\prime}}\in A_{mp_i^{\prime}}$ and $\mu_{\vec{\omega}}\in A_{mq}$.

\subsection{The Campanato spaces}
Let $0<q<\infty$ and $-n/q<\alpha<1$. A locally integrable function $f$ is said to belong to Campanato space $\mathcal{C}_{\alpha,q}$ if there exists a constant
$C > 0$ such that for any cube $Q\subset \mathbb{R}^n$,
$$\frac{1}{|Q|^{\alpha/n}}\bigg(\frac{1}{|Q|}\int_{Q}|f(x)-m_{Q}(f)|^{q}dx\bigg)^{1/q}\leq C,$$
where $m_{Q}(f)=\frac{1}{|Q|}\int_{Q}f(x)dx$ and the minimal constant $C$ is defined by $\|f\|_{\mathcal{C}_{\alpha,q}(\mathbb{R}^n)}$.

Campanato spaces are a useful tool in the regularity theory of PDEs due to their better structures,
which allows us to give an integral characterization of the spaces of H\"{o}lder continuous functions when $0<\alpha<1$. The Lipschitz (H\"{o}lder) and Campanato spaces are related by the following equivalences:
$$\|f\|_{Lip_{\alpha}(\mathbb{R}^n)}:=\sup_{x,h\in \mathbb{R}^n,h\neq 0}\frac{|f(x+h)-f(x)|}{|h|^{\alpha}}\approx \|f\|_{\mathcal{C}_{\alpha,q}},\quad 0<\alpha<1.$$
The equivalence can be found in \cite{DS1984} for $q=1$, \cite{JTW1983} for $1<q<\infty$ and \cite{WZTmn} for $0<q<1$.

Specially, $\mathcal{C}_{0,q}(\mathbb{R}^n)=BMO(\mathbb{R}^n)$, the spaces of bounded mean oscillation. The crucial property of $BMO$ functions is the John-Nirenberg inequality \cite{JN1961},
$$|\{x\in Q: |f(x)-m_{Q}(f)|>\lambda\}|\leq c_{1}|Q|e^{-\frac{c_{2}\lambda}{\|f\|_{BMO}}(\mathbb{R}^n)},$$
where $c_{1}$ and $c_{2}$ depend only on the dimension. A well-known immediate corollary of the John-Nirenberg inequality as follows:
$$\|f\|_{BMO}\approx \sup_{Q}\frac{1}{|Q|}\Big(\int_{Q}|f(x)-m_{Q}(f)|^{p}dx\Big)^{1/p},$$
for all $1<p<\infty$. In fact, the equivalence also holds for $0<p<1$ (see \cite{WZTams}).

\subsection{Hardy spaces}
The theory of Hardy spaces is vast and complicated, it has been systematically developed and plays an important role in harmonic analysis and PDEs. A bounded tempered distribution $f$ is in the Hardy space $H^{s}(\mathbb{R}^n)$ if the Poisson maximal function
$$M(f;P)=\sup_{t>0}|(P_{t}*f)(x)|$$
lies in $L^{\rho}(\mathbb{R}^n)$. We first recall the atomic decomposition of Hardy spaces.
Let $0<s\leq 1\leq q\leq \infty, \rho\neq q$ and the nonnegtive integer $l\geq [n(\frac{1}{\rho}-1)]$ ($[x]$ indicates the integer part of $x$). A function $a\in L^{q}(\mathbb{R}^n)$ is called a $(\rho,q,l)$ atom for $H^\rho(\mathbb{R}^n)$ if there exists a cube $Q$ such that
\begin{itemize}
  \item [\rm (i)]  $a$ is supported in $Q$;
  \item [\rm (ii)] $\|a\|_{L^{q}(\mathbb{R}^n)}\leq |Q|^{\frac{1}{q}-\frac{1}{\rho}}$;
  \item [\rm (iii)] $\int_{\mathbb{R}^n} a(x)x^{\alpha}dx=0$ for all multi-indices $\alpha$ with $0\leq |\alpha|\leq l$.
  \end{itemize}
Here, $(\rm i)$ means that an atom must be a function with compact support,
$(\rm ii)$ is the size condition of atoms, and $(\rm iii)$ is called the cancellation moment
condition. Notice that since $\frac{n}{n+1}<\rho\leq 1$ in our case, then we have $l=0$.

The atomic Hardy space $H^{\rho,q,l}(\mathbb{R}^n)$ is defined by
\begin{align*}
H^{\rho,q,l}(\mathbb{R}^n)=\Big\{f\in \mathcal{S}': f=^{S'}\sum_{k}\lambda_{k}a_{k}(x), \text{each} ~~a_{k} ~~\text{is a} ~~(\rho,q,l)-\text{atom}, \text{and}~~ \sum_{k}|\lambda|^{\rho}<\infty\Big\}.
\end{align*}
Setting $H^{\rho,q,l}(\mathbb{R}^n)$ norm of $f$ by
$$\|f\|_{H^{\rho,q,l}(\mathbb{R}^n)}=\inf \big(\sum_{k}|\lambda_k|^{\rho}\big)^{1/\rho},$$
where the infimum is taken over all decompositions of $f=\sum_{k}\lambda_{k}a_{k}$ above. Note that $H^{\rho,q,l}(\mathbb{R}^n)=H^{\rho}(\mathbb{R}^n)$ was proved by Cofiman \cite{C1974} for $n=1$ and Latter \cite{L1978} for $n>1$. This indicates that each element in $H^{\rho}(\mathbb{R}^n)$ can be decomposed into a sum of atoms in certain way.

\section{Some auxiliary lemmas}\label{s3}

\qquad In this section, we will proceed with the proof of the following auxiliary lemmas, which we need in order to prove our main results.

Next, we recall a technical lemma about certain $H^{\rho}(\mathbb{R}^n)$ functions.

\begin{lemma} \label{lem-main1}
     Let $\frac{n}{n+1}<\rho<1$ and $f$ be a function satisfying the following estimates:
\begin{itemize}
  \item [\rm (i)]  $\int_{\mathbb{R}^n}f(x)dx=0;$
  \item [\rm (ii)] there exist balls $B_1=B(x_1,r)$ and $B_2=B(x_2,r)$ for some $x_1,x_2\in \mathbb{R}^n$ and $r>0$ such that
     $$|f(x)|\le h_1(x)\chi_{B_1}(x)+h_2(x)\chi_{B_2}(x),$$
     where $\|h_{i}\|_{L^{q}(\mathbb{R}^n)}\leq C|B_{i}|^{1/q-1/\rho}$ with $1<q\leq \infty$;
  \item [\rm (iii)] $|x_1-x_2|= Nr$.
  \end{itemize}
  Then, $f\in H^\rho(\mathbb{R}^n)$ and there exists a positive constant $C$ independent of $x_1,x_2,r$ such that
     $$  \|f \|_{H^\rho(\mathbb{R}^n)} \leq CN^{n(1/\rho-1)}.$$
      \end{lemma}
\begin{proof}
Since $|f(x)|\le h_1(x)\chi_{B_1}(x)+h_2(x)\chi_{B_2}(x),$ then we can write
$$f:=f_1+f_2 \quad \text{with}\quad  |f_{i}|\leq h_{i}$$
and supp $f_{i}\subset B_i$ for $i=1,2.$ We will show that $f$ has the following atomic decomposition
\begin{align}\label{lem1-1}
    f=\sum_{i=1}^{2}\sum_{j=1}^{J_0+1}\lambda_i^ja_i^j,
\end{align}
where $J_0$ is the smallest integer larger than $\log N$.

Let $i=1,2$, $\alpha_i^0=f_i$ and
\begin{equation}\label{alpha1}
\alpha_i^j=m_{2^{j}B_i}(f_i),\qquad j=1,\cdots,J_0,
\end{equation}
with $|\alpha_{i}^{j}|=\frac{|B_{i}|}{|2^{j}B_{i}|}m_{B_{i}}(f_{i})\leq 2^{-jn}|B_{i}|^{-1/s}$. We also write
\begin{equation}\label{f1}
f_i^j=\alpha_i^{j-1}\chi_{2^{j-1}B_i}-\alpha_i^{j}\chi_{2^{j}B_i}.
\end{equation}
Which shows that $f_i^j$ satisfies the cancellation moment condition, and
\begin{equation}\label{alpha2}
\|f_{i}^{j}\|_{L^{q}(\mathbb{R}^n)}\lesssim |\alpha_{i}^{j}||2^{j}B_{i}|^{1/q}\lesssim 2^{-jn/q'}|B_{i}|^{1/q-1/\rho}.
\end{equation}

Now, we return to the proof. By the direct compute, we have
\begin{align*}
	f=\sum_{i=1}^{2}f_{i}-\sum_{i=1}^{2}\alpha_i^1\chi_{2B_i}+\sum_{i=1}^{2}\alpha_i^1\chi_{2B_i}=\sum_{i=1}^{2}(f_i-\alpha_i^1\chi_{2B_i})+\sum_{i=1}^{2}\alpha_i^1\chi_{2B_i}.
\end{align*}
From the definitions of \eqref{alpha1} and \eqref{f1}, it follows that $$f=\sum_{i=1}^{2}f_i^1+\sum_{i=1}^{2}\alpha_i^1\chi_{2B_i}.$$
Continuing in this process with
\begin{equation}\label{f2}
\begin{aligned}
	f=&\sum_{i=1}^{2}f_i^1+\sum_{i=1}^{2}\alpha_i^1\chi_{2B_i}-\alpha_i^2\chi_{2^2B_i}+\alpha_i^2\chi_{2^2B_i}\\
	=&\sum_{i=1}^{2}f_i^1+\sum_{i=1}^{2}f_i^2+\sum_{i=1}^{2}\alpha_i^2\chi_{2^2B_i}-\alpha_i^3\chi_{2^3B_i}+\alpha_i^3\chi_{2^3B_i}\\
	=&\sum_{i=1}^2(\sum_{j=1}^{J_0}f_i^j)+\sum_{i=1}^2\alpha_i^{J_0}\chi_{2^{J_0}B_i}.
\end{aligned}
\end{equation}
Now for $j=1,2,\cdots,J_0$, let
$$a_i^j=\frac{f_i^j}{\|f_i^j\|_{L^{q}(\mathbb{R}^n)}}|2^jB_i|^{1/q-1/\rho}.$$
Then, $a_i^j$ is a $(\rho,q,0)$-atom. In fact, it is easy to see
$$supp\ a_i^j\subseteq B(x_i,2^jr), \qquad \|a_i^j\|_{L^{q}(\mathbb{R}^n)}=|2^jB_i|^{1/q-1/\rho},$$
and
$$\int_{\mathbb{R}^n} f_i^j(x)dx=\int_{\mathbb{R}^n} (\alpha_i^{j-1}\chi_{B(x_i,2^{j-1}r)}-\alpha_i^{j}\chi_{B(x_i,2^{j}r)} )dx=0,$$
which gives us that
$$\int a_i^j(x)dx=\frac{|2^jB_i|^{1/q-1/\rho}}{\|f_i\|_{L^{q}(\mathbb{R}^n)}}\int_{\mathbb{R}^n} f_i^j(x)dx=0.$$

Therefore, the \eqref{f2} is rewritten as
\begin{equation}\label{f3}
\begin{aligned}
f=\sum_{i=1}^{n} \Big(\sum_{j=1}^{J_0}||f_i^j||_{L^{q}(\mathbb{R}^n)}|2^jB_i|^{1/\rho-1/q}a_i^j \Big)
+\sum_{i=1}^{2}\alpha_i^{J_0}\chi_{B(x_i,2^{J_0}r)}.
\end{aligned}
\end{equation}
Next, we will estimate
\begin{equation}\label{f4}
\begin{aligned}
\sum_{i=1}^2\alpha_i^{J_0}\chi_{2^{J_0}B_i}.
\end{aligned}
\end{equation}
In order to achieve this, we set $$\alpha^{J_0}=\frac{|B(x_1,r)|}{|B(\frac{x_1+x_2}{2},2^{J_0}r)|}m_{B_{1}}(f_1). $$
Notice that, since $f=f_1+f_2$, and $\int_{\mathbb{R}^n} f(x)dx=0,$ we also have
$$\alpha^{J_0}=-\frac{|B(x_2,r)|}{|B(\frac{x_1+x_2}{2},2^{J_0}r)|}m_{B_{2}}(f_2). $$
Combining with \eqref{f4} and the term $\alpha^{J_0}\chi_{B(\frac{x_1+x_2}{2},2^{J_0+1}r)}$, we get
\begin{equation}\label{f5}
\begin{aligned}
	\sum_{i=1}^{2}\alpha_i^{J_0}\chi_{2^{J_0}B_i}&=f_1^{J_0+1}+f_2^{J_0+1},
\end{aligned}
\end{equation}
where
$$f_i^{J_0+1}=\alpha_i^{J_0}\chi_{2^{J_{0}}B_i}-\alpha_i^{J_0}\chi_{B(\frac{x_1+x_2}{2},2^{J_0+1}r)}.$$
Meanwhile, it is easy to see that
$$a_i^{J_0+1}=\frac{f_i^{J_0+1}}{||f_i^{J_0+1}||_{L^{q}(\mathbb{R}^n)}}|B(\frac{x_1+x_2}{2},2^{J_0+1}r)|^{1/q-1/s},$$
is a $(\rho,q,0)$-atom.

Together with \eqref{f3} and \eqref{f5}, we have
$$f=\sum_{i=1}^{2}\sum_{j=1}^{J_0+1}\lambda_i^j a_i^j,$$
where
$$\lambda_i^j= \Bigg\{\begin{matrix}
	&||f_i^j||_{L^{q}(\mathbb{R}^n)}|2^jB_i|^{1/s-1/q}, & j=1,\cdots,J_0\\
    &\\
	&||f_i^j||_{L^{q}(\mathbb{R}^n)}|B(\frac{x_1+x_2}{2},2^{J_0+1}r)|^{1/s-1/q}, & j=J_0+1.
\end{matrix} .$$
Note that, for $j=1,\cdots,J_0+1$, by the inequality \eqref{alpha2}, we then have
\begin{align*}
	|\lambda_i^j|=||f_i^j||_{L^{q}(\mathbb{R}^n)}|2^jB_i|^{1/s-1/q}\leq C2^{jn(1/\rho-1)}.
\end{align*}
It implies that $f\in H^\rho(\mathbb{R}^n)$ and
\begin{align*}
	||f||^s_{H^s{(\mathbb{R}^n})}&\lesssim \sum_{i=1}^{2}\sum_{j=1}^{J_0+1}|\lambda_i^j|^\rho
	\lesssim \sum_{i=1}^{2}\sum_{j=1}^{J_0+1}2^{nj(1-\rho)}
	\lesssim \sum_{i=1}^{2}2^{(J_0+1)n(1-\rho)}	\lesssim N^{n(1-\rho)}.
\end{align*}
This finishes the proof of Lemma \ref{lem-main1}.
\end{proof}
\begin{remark}
The general case $h_{1}(x)=C_{1}$ and $h_{2}(x)=C_{2}$ was proved by Kuffner in \cite{K2018}. However, this case does not apply to the factorization theorems with the genuinely multilinear weights.
\end{remark}

For $\rho=1$,  we record the following lemma, which is also useful in the proof of our main results and is proved by Wang and Zhu in \cite{WZ2021}.
\begin{lemma} \label{lem-main11}
     Let $f$ be a function satisfying the following estimates:
\begin{itemize}
  \item [\rm (i)]  $\int_{\mathbb{R}^n}f(x)dx=0;$
  \item [\rm (ii)] there exist balls $B_1=B(x_1,r)$ and $B_2=B(x_2,r)$ for some $x_1,x_2\in \mathbb{R}^n$ and $r>0$ such that
     $$|f(x)|\le h_1(x)\chi_{B_1}(x)+h_2(x)\chi_{B_2}(x),$$
     where $\|h_{i}\|_{L^{q}(\mathbb{R}^n)}\leq C|B_{i}|^{1/q-1}$ with $1<q\leq \infty$;
  \item [\rm (iii)] $|x_1-x_2|= Nr$.
  \end{itemize}
  Then, $f\in H^1(\mathbb{R}^n)$ and there exists a positive constant $C$ independent of $x_1,x_2,r$ such that
     $$  \|f \|_{H^1(\mathbb{R}^n)} \leq C\log N.$$
      \end{lemma}

\begin{lemma} \label{lem-main2}
Suppose  $1\le l\le m$, $1< p, p_{1},\dots,p_{m}<\infty$ with
$$\frac{1}{p_{1}}+ \dots +\frac{1}{p_{m}}=\frac{1}{p}$$
and $\vec{\omega}\in A_{\vec{P}}$. For any $g\in L^{p^{\prime}}(\nu^{1-p^{\prime}}_{\vec{\omega}})$ and $h_{i}\in L^{p_{i}}(\omega_i), i=1,2,\dots ,m$, we have
\begin{align*}	
\|\Pi_{l} (g,h_{1},\ldots,h_{m} ) \|_{H^{1} (\mathbb{R}^n )}\lesssim \|g\|_{L^{p^{\prime}}(\nu^{1-p^{\prime}}_{\vec{\omega}})} \|h_{1} \|_{L^{p_{1}}(\omega)_1}\cdots \|h_{m} \|_{L^{p_{m}}(\omega_m)}.
\end{align*}
\end{lemma}
\begin{proof}
	Note that $T$ is bounded from $L^{p_{1}}(\omega_{1})\times \cdots\times \cdots L^{p_{m}}(\omega_{m})$ to $L^{p}(\nu_{\vec{\omega}})$. For any $g\in L^{p^{\prime}}(\nu^{1-p^{\prime}}_{\vec{\omega}})$ and $h_{i}\in L^{p_{i}}(\omega_i), i=1,2,\dots ,m$, we have
\begin{equation}\label{L1}
\begin{aligned}
\int_{\mathbb{R}^n}|g(x)T(h_{1},\cdots,h_{m})(x)|dx&=\int_{\mathbb{R}^n}|g(x)|\nu_{\vec{\omega}}(x)^{-\frac{1}{p}}\cdot |T(h_{1},\cdots,h_{m})(x)|\nu_{\vec{\omega}}(x)^{\frac{1}{p}}dx\\
&\leq\|g\|_{L^{p^{\prime}}(\nu_{\vec{\omega}}^{1-p'})}\|T(h_{1},\cdots,h_{m})\|_{L^{p}(\nu_{\vec{\omega}})}\\
&\lesssim \|g\|_{L^{p^{\prime}}(\nu_{\vec{\omega}}^{1-p'})}\prod_{i=1}^{m}\|h_{i}\|_{L^{p_{i}}(\omega_i)}.
\end{aligned}
\end{equation}
On the other hand, the directly calculation gives us that
\begin{align*}
	\frac{1}{p_l^{\prime}}=\sum_{j\ne l}\frac{1}{p_j}+\frac{1}{p'} \qquad and \qquad \omega_{l}^{1-p'_{l}}=\nu_{\vec{\omega}}^{-\frac{p_{l}}{p}}\prod_{i\neq l}\omega_{i}^{\frac{p_{l}}{p_i}},
\end{align*}
which implies that
\begin{align*}
T^*_l: L^{p_1}(\omega_1)\times\cdots\times L^{p_{l-1}}(\omega_{l-1})\times L^{p'}(\nu_{\vec{\omega}}^{1-p'})\times L^{p_{l+1}}(\omega_{l+1})\times\cdots\times L^{p_m}(\omega_m) \rightarrow L^{p_l^{\prime}}(\omega_{l}).
\end{align*}
and
\begin{equation}\label{L2}
\begin{aligned}
&\int_{\mathbb{R}^n}\Big|h_{l}(x)T^*_l(h_{1},\cdots,h_{l-1},g,h_{l+1},\cdots,h_{m})(x)\Big|dx\\
&=\int_{\mathbb{R}^n}|h_{l}(x)|\omega_{l}(x)^{\frac{1}{p_{l}}}\cdot\big|T^*_l(h_{1},\cdots,h_{l-1}(x),g(x),h_{l+1}(x),\cdots,h_{m})(x)\big|\omega_{l}(x)^{-\frac{1}{p_{l}}}dx\\
&\lesssim \|g\|_{L^{p^{\prime}}(\nu_{\vec{\omega}}^{1-p'})}\prod_{i=1}^{m}\|h_{i}\|_{L^{p_{i}}(\omega_i)}.
\end{aligned}
\end{equation}
By \eqref{L1} and \eqref{L2}, we conclude that ${\textstyle\prod_{l}}(g,h_1,\cdots,h_m)(x)\in L^{1}(\mathbb{R}^n)$. On the other hand,
$$\int_{\mathbb{R}^n}{\textstyle\prod_{l}}(g,h_1,\cdots,h_m)(x)dx=0.$$
Hence, for $b\in BMO(\mathbb{R}^n)$,
\begin{align*}
	 &\Big|\int_{\mathbb{R}^n}b(x){\textstyle\prod_{l}}(g,h_1,\cdots,h_m)(x)dx \Big|\\
	&= \Big|\int_{\mathbb{R}^n}g(x)[b,T]_l(h_1,\cdots,h_m)(x)dx \Big|\\
	&= \Big|\int_{\mathbb{R}^n}g(x)\nu_{\vec{\omega}}(x)^{-1/p}[b,T]_l(h_1,\cdots,h_m)(x)\nu_{\vec{\omega}}(x)^{1/p}dx \Big|\\
	&\le||g||_{L^{p'}(\nu_{\vec{\omega}}^{1-p'})}\cdot||[b,T]_l(h_1,\cdots,h_m)||_{L^p(\nu_{\vec{\omega}})}\\
	&\lesssim ||h_1||_{L^{p_1}(\omega_1)}\cdots||h_m||_{L^{p_m}(\omega_m)}||g||_{L^{p'}(\nu_{\vec{\omega}}^{1-p'})}||b||_{BMO(\mathbb{R}^n)}.
\end{align*}
Therefore, $\textstyle\prod_{l}(g,h_1,\cdots,h_m)$ is in $H^1(\mathbb{R}^n)$ with
$$||\textstyle\prod_{l}(g,h_1,\cdots,h_m)||_{H^1(\mathbb{R}^n)}\lesssim ||g||_{L^{p'}(\nu_{\vec{\omega}}^{1-p'})}||h_1||_{L^{p_1}(\omega_1)}\cdots||h_m||_{L^{p_m}(\omega_m)}.$$
Thus, the proof of Lemma \ref{lem-main2} is completed.
\end{proof}

\begin{lemma} \label{lem-main22}
Suppose  $1\le l\le m$, $\frac{1}{n+1}<\rho<1$, $1< p_{1},\dots,p_{m},q<\infty$ with
$$\frac{1}{p_{1}}+ \dots +\frac{1}{p_{m}}+1=\frac{1}{\rho}+\frac{1}{q}.$$
and $\vec{\omega}\in A_{\vec{P},q}$. For any $g\in L^{q^{\prime}}(\mu^{1-q^{\prime}}_{\vec{\omega}})$ and $h_{i}\in L^{p_{i}}(\omega_i^{p_{i}}), i=1,2,\dots ,m$, we have
\begin{align*}	
\|\Pi_{l} (g,h_{1},\ldots,h_{m} ) \|_{H^{\rho} (\mathbb{R}^n )}\lesssim \|g\|_{L^{q^{\prime}}(\mu^{1-q^{\prime}}_{\vec{\omega}})} \|h_{1} \|_{L^{p_{1}}(\omega_1^{p_{1}})}\cdots \|h_{m} \|_{L^{p_{m}}(\omega_m^{p_{m}})}.
\end{align*}
\end{lemma}
\begin{proof}

Let $W_{i}=\omega_{i}^{p_{i}}$, $\vec{W}=(W_{1},\cdots,W_{m})$ and $\frac{1}{p}=\frac{1}{p_{1}}+ \dots +\frac{1}{p_{m}}$. Therefore,
$$\nu_{\vec{W}}=\prod_{i=1}^{m} W_{i}^{\frac{p}{p_i}}=\prod_{i=1}^{m}\omega_{i}^{p}=\mu_{\vec{\omega}}^{p/q}.$$
Thus, $\vec{\omega}\in A_{\vec{P},q}$ is equivalent to
$$\sup_{B}\Big(\frac{1}{|B|}\int_{B}\nu_{\vec{W}}^{\frac{q}{p}}dx\Big)^{\frac{1}{q}}\prod_{i=1}^{m}\Big(\frac{1}{|B|}\int_{B}W_{i}^{1-p'_{i}}dx\Big)^{\frac{1}{p'_{i}}}<\infty.$$
In other words, $\vec{W}\in A_{\vec{P}}$ by H\"{o}lder inequality. Which shows that $T$ is bounded from $L^{p_{1}}(\omega_{1}^{p_{i}})\times \cdots\times \cdots L^{p_{m}}(\omega_{m}^{p_{m}})$ to $L^{p}(\mu_{\vec{\omega}}^{p/q})$. For any $g\in L^{q^{\prime}}(\mu^{1-q^{\prime}}_{\vec{\omega}})$ and $h_{i}\in L^{p_{i}}(\omega_i^{p_{i}}), i=1,2,\dots ,m$, we have
Combining the above estimates and Young inequality with $\frac{1}{p}+\frac{1}{q'}=\frac{1}{\rho}$, we get
\begin{equation}\label{L1}
\begin{aligned}
&\Big(\int_{\mathbb{R}^n}|g(x)T(h_{1},\cdots,h_{m})(x)|^{\rho}dx\Big)^{1/\rho}\\
&=\int_{\mathbb{R}^n}|g(x)|\mu_{\vec{\omega}}(x)^{-\frac{1}{q}}\cdot |T(h_{1},\cdots,h_{m})(x)|\mu_{\vec{\omega}}(x)^{\frac{1}{q}}dx\\
&\leq\|g\|_{L^{q^{\prime}}(\mu_{\vec{\omega}}^{1-q'})}\|T(h_{1},\cdots,h_{m})\|_{L^{p}(\mu_{\vec{\omega}}^{p/q})}\\
&\lesssim \|g\|_{L^{q^{\prime}}(\mu_{\vec{\omega}}^{1-q'})}\prod_{i=1}^{m}\|h_{i}\|_{L^{p_{i}}(\omega_i^{p_{i}})}.
\end{aligned}
\end{equation}
Meanwhile, we further write
$$\frac{1}{\tilde{p}}:=\sum_{j\ne l}\frac{1}{p_j}+\frac{1}{q'}.$$
the directly calculation gives us that
\begin{align*}
	\frac{1}{\tilde{p}}+\frac{1}{p_{l}}=\frac{1}{\rho} \qquad and \qquad
\omega_{l}^{-\tilde{p}}=\big(\mu_{\vec{\omega}}^{1-q'}\big)^{\tilde{p}/q'}\prod_{i\neq l}\omega_{i}^{\tilde{p}},
\end{align*}
which implies that
\begin{align*}
T^*_l: L^{p_1}(\omega_1^{p_{1}})\times\cdots\times L^{p_{l-1}}(\omega_{l-1}^{p_{l-1}})\times L^{q'}(\mu_{\vec{\omega}}^{1-q'})\times L^{p_{l+1}}(\omega_{l+1}^{p_{l+1}})\times\cdots\times L^{p_m}(\omega_m^{p_{m}}) \rightarrow L^{\tilde{p}}(\omega_{l}^{-\tilde{p}}).
\end{align*}
By Young inequality with $\frac{1}{\tilde{p}}+\frac{1}{p_{l}}=\frac{1}{\rho}$, we get
\begin{equation}\label{L2}
\begin{aligned}
&\Big(\int_{\mathbb{R}^n}\big|h_{l}(x)T^*_l(h_{1},\cdots,h_{l-1},g,h_{l+1},\cdots,h_{m})(x)\big|^{\rho}dx\Big)^{1/\rho}\\
&=\Big(\int_{\mathbb{R}^n}\big|h_{l}(x)|\omega_{l}(x)\cdot T^*_l(h_{1},\cdots,h_{l-1},g,h_{l+1},\cdots,h_{m})(x)\omega_{l}(x)^{-1}\big|^{\rho}dx\Big)^{1/\rho}\\
&\leq \|h_{l}\|_{L^{p_{i}}(\omega_l^{p_{l}})}  \|T^*_l(h_{1},\cdots,h_{l-1},g,h_{l+1},\cdots,h_{m})\|_{L^{\tilde{p}}(\omega_{l}^{-\tilde{p}})}\\
&\lesssim \|g\|_{L^{p^{\prime}}(\nu_{\vec{\omega}}^{1-p'})}\prod_{i=1}^{m}\|h_{i}\|_{L^{p_{i}}(\omega_i)}.
\end{aligned}
\end{equation}
By \eqref{L1} and \eqref{L2}, we conclude that ${\textstyle\prod_{l}}(g,h_1,\cdots,h_m)(x)\in L^{\rho}(\mathbb{R}^n)$. On the other hand,
$$\int_{\mathbb{R}^n}{\textstyle\prod_{l}}(g,h_1,\cdots,h_m)(x)dx=0.$$
Due to
$$\big|[b,T]_{l}(f_{1},\cdots,f_{m})(x)\big|\lesssim \|b\|_{Lip_{\alpha}(\mathbb{R}^n)}I_{\alpha,m}(|f_{1}|,\cdots,|f_{m}|)(x),$$
where $b\in Lip_{\alpha}(\mathbb{R}^n)$ and $I_{\alpha,m}$ is the multilinear fractional integral operator
$$I_{\alpha,m}(f_{1},\cdots,f_{m})(x)=\int_{(\mathbb{R}^n)^m}\frac{f_{1}(y_{1}\cdots f_{m}(y_{m}))}{(|x-y_{1}|+\cdots+|x-y_{m}|)^{mn-\alpha}}dy_{1}\cdots dy_{m}.$$
Now, we return to the proof. For $b\in Lip_{\alpha}(\mathbb{R}^n)$ with $\alpha=n(\frac{1}{\rho}-1)$,
\begin{align*}
	 &\Big|\int_{\mathbb{R}^n}b(x){\textstyle\prod_{l}}(g,h_1,\cdots,h_m)(x)dx \Big|\\
	&= \Big|\int_{\mathbb{R}^n}g(x)[b,T]_l(h_1,\cdots,h_m)(x)dx \Big|\\
	&= \Big|\int_{\mathbb{R}^n}g(x)\mu_{\vec{\omega}}^{-1/q}[b,T]_l(h_1,\cdots,h_m)(x)\mu_{\vec{\omega}}^{1/q}dx \Big|\\
	&\lesssim \|g\|_{L^{q'}(\nu_{\vec{\omega}}^{1-q'})}\cdot\|I_{\alpha}(|h_1|,\cdots,|h_m)|\|_{L^q(\mu_{\vec{\omega}})}\|b\|_{Lip_{\alpha}(\mathbb{R}^n)}\\
	&\lesssim \|h_1\|_{L^{p_1}(\omega_1^{p_{1}})}\cdots\|h_m\|_{L^{p_m}(\omega_m^{p_{m}})}\|g\|_{L^{q'}(\mu_{\vec{\omega}}^{1-p'})}\|b\|_{Lip_{\alpha}(\mathbb{R}^n)}.
\end{align*}
Therefore, $\textstyle\prod_{l}(g,h_1,\cdots,h_m)$ is in $H^\rho(\mathbb{R}^n)$ with
$$\|\textstyle\prod_{l}(g,h_1,\cdots,h_m)\|_{H^\rho(\mathbb{R}^n)}\lesssim \|g\|_{L^{p'}(\mu_{\vec{\omega}}^{1-p'})}\|h_1\|_{L^{p_1}(\omega_1^{p_{1}})}\cdots\|h_m\|_{L^{p_m}(\omega_m^{p_{m}})}.$$
Thus, we complete the proof of Lemma \ref{lem-main22}.
\end{proof}

\begin{lemma} \label{lem-main3}
Let $1\le l\le m$, $1<p, p_{1},\dots,p_{m}<\infty$, $\frac{1}{p_1}+\dots+\frac{1}{p_m} = \frac{1}{p}$ and $\vec{\omega}\in A_{\vec{P}}$. For every $H^{1}(\mathbb{R}^n)$-atom $a(x)$,
 there exists $g\in L^{p^{\prime}}(\nu^{1-p^{\prime}}_{\vec{\omega}})$ and $h_{i}\in L^{p_{i}}(\omega_i), i=1,2,\dots ,m$ and a large positive number $N$ (depending only on $\varepsilon$) such that:
 $$   \| a- {\textstyle \prod_{l}}(g,h_1,h_2,\dots ,h_m)  \|_{H^1(\mathbb{R}^n)}<\varepsilon $$ and that
 \begin{align*}
 	\|g\|_{L^{p^{\prime}}(\nu^{1-p^{\prime}}_{\vec{\omega}})} \|h_{1} \|_{L^{p_{1}}(\omega_1)}\cdots \|h_{m} \|_{L^{p_{m}}(\omega_m)}\lesssim N^{mn(1+L)}.
 \end{align*}
 \end{lemma}
 \begin{proof}
 Let $a(x)$ be an $H^{1}(\mathbb{R}^n)$-atom, supported in $B(x_0,r)$, satisfying that
 $$ \int_{\mathbb{R}^n}a(x)dx=0\qquad and \qquad   \|a\|_{L^{\infty}(\mathbb{R}^n)}\le |B(x_{0},r)|^{-1}.$$
 Fix $1\le l \le m$. Now select $y_{l}\in \mathbb{R}^n$ so that $y_{l,i}-x_{0,i}=\frac{Nr}{\sqrt{n}}$, where $x_{0,i}$(reps.$y_{l,i}$) is the $i$-th coordinate of $x_{0}$(reps.$y_{l}$) for
 $i=1,2,\dots ,n$. Note that for this $y_{l}$, we have $  | x_0-y_l   | =Nr$. Similar to the relation of $x_0$ and $y_l$, we choose $y_1$ such that $y_0$ and $y_1$
 satisfies the same relationship as $x_0$ and $y_l$ do.  Then by induction we choose $y_2,\dots ,y_{l-1},y_{l+1},\dots ,y_m$.
We write $B_i=B(y_i,r)$, $\nu_i=\omega_i^{1-p_i^{\prime}}$ and set
 \begin{equation*}
\begin{split}  	
&g(x):=\big(\frac{|B_{l}|\nu_{\vec{\omega}}(x)}{\nu_{\vec{\omega}}(B_{l})}\big)^{\frac{1}{p}}\chi_{B_l}(x),\\ &g(B_{l}):=\int_{B(y_l,r)}g(z_{l})dz_{l},\\
&h_j(x):=\nu_{j}(x)\chi_{B_j}(x),\qquad j\ne l,\\
&h_l(x)=\frac{a(x)}{T_l(h_1,\dots,h_{l-1},g,h_{l+1},\dots,h_m)(x_0)}\frac{g(B_{l})\nu_{l}(x)}{\nu_{l}(B_{l})}\chi_{B_l}(x).\\
 	\end{split}
    \end{equation*}

From the definitions of the functions $g(x)$, we obtain that $supp\ g=B(y_l,r)$ and
   \begin{equation}\label{lem3-2}
    \begin{aligned}
    \|g\|_{L^{p^{\prime}}(\nu_{\vec{\omega}}^{1-p^{\prime}})}=\Big(\frac{|B_l|}{\nu_{\vec{\omega}}(B_l)}\Big)^{1/p}|B_l|^{1/p^{\prime}}=\frac{|B_l|}{\nu_{\vec{\omega}}(B_l)^{1/p}}
    \end{aligned}
    \end{equation}
For the specific choice of the functions
    $h_1,\dots ,h_{l-1},g,h_{l+1},\dots,h_m$ as above, there exists a positive constant $C$ such that
   \begin{equation}\label{lem3-1}
    \begin{aligned}
    	&|T_{l}(h_1,\dots,h_{l-1},g,h_{l+1},\dots,h_m)(x_0)|\\
    	&\geq C\int_{B_1\times\cdots\times B_m}\frac{g(z_{l})\prod_{j\ne l}\nu_{j}(z_j)}{(|x_0-z_1|+\cdots+|x_0-z_m|)^{mn}}dz_1\cdots dz_m\\
    	&\ge C(Nr)^{-mn}g(B_{l})\prod_{j\ne l}\nu_j(B_j).
    \end{aligned}
    \end{equation}
    Which shows that
   \begin{equation}\label{lem3-3}
    \begin{aligned}
    \|h_j\|_{L^{p_j}(\omega_j)}\lesssim \nu_j(B_j)^{\frac{1}{p_j}}
    \end{aligned}
    \end{equation}
    for $i=1,\dots,l-1,l+1,\dots,m$. Also, the inequality \eqref{lem3-1} and $\nu_{l}^{p_{l}}\omega_{l}=\nu_{l}$ give us that
   \begin{equation}\label{lem3-4}
    \begin{aligned}
    	 \|h_{l} \|_{L^{p_{l}}(\omega_l)}
    	& = \frac{1}{ | T_{l} (h_{1}, \ldots, h_{l-1}, g, h_{l+1}, \ldots, h_{m} ) (x_{0} ) |}\frac{g(B_{l})}{\nu_{l}(B_{l})}\|a\|_{L^{p_{l}}(\nu_{l})} \\
    	&\lesssim r^{-n}(Nr)^{mn}\nu_{l}(B_{l})^{-1/p'_{l}}\prod_{j\neq l}\nu_j(B_j)^{-1}.
    \end{aligned}
    \end{equation}
Combining the estimates \eqref{lem3-2}, \eqref{lem3-3} and \eqref{lem3-4}, we arrive at
   \begin{equation}\label{lem3-5}
    \begin{aligned}
    	\|g\|_{L^{p^{\prime}}(\nu^{1-p^{\prime}}_{\vec{\omega}})} \|h_{1} \|_{L^{p_{1}}(\omega_1)}\cdots \|h_{m} \|_{L^{p_{m}}(\omega_m)}
    	&\lesssim (Nr)^{mn}\nu_{\vec{\omega}}(B_l)^{-1/p}\prod_{j=1}^{m}\nu_j(B_j)^{-1/p_j^{\prime}}.
    \end{aligned}
    \end{equation}
From the inequality \eqref{weight-eq1}, then
     $${\nu_j(B_l)}\lesssim {\nu_j((N+1)B_j)}\lesssim N^{nL}\nu_j(B_j),$$
which shows that
$$\nu_j(B_j)^{-1/p'_{j}}\lesssim N^{nL}\nu_j(B_l)^{-1/p'_{j}}.$$
Along with \eqref{lem3-5} and the fact that $|B_l|^{m}\leq \nu_{\vec{\omega}}(B_l)^{1/p}\prod_{j=1}^{m}\nu_j(B_l)^{1/p_j^{\prime}}$, we have
   \begin{equation}\label{wzs}
    \begin{aligned}
    \|g\|_{L^{p^{\prime}}(\nu^{1-p^{\prime}}_{\vec{\omega}})} \|h_{1} \|_{L^{p_{1}}(\omega_1)}\cdots \|h_{m} \|_{L^p_{m}(\omega_m)}\lesssim N^{mn(1+L)}.
        \end{aligned}
    \end{equation}
    Next, we have
	\begin{equation*}
	\begin{split}  
		&a(x)- {\textstyle \prod_{l}}(g,h_1,h_2,\dots ,h_m)(x)\\
		&=a(x)-  [h_l T_l^*(h_1,\dots,h_{l-1},g,h_{l+1},\dots,h_m)-gT(h_1,\dots,h_m)(x)  ] \\
		&=a(x)\frac{T_l^*(h_1,\dots,h_{l-1},g,h_{l+1},\dots,h_m)(x_0)-T_l^*(h_1,\dots,h_{l-1},g,h_{l+1},\dots,h_m)(x)}{T_l^*(h_1,\dots,h_{l-1},g,h_{l+1},\dots,h_m)(x_0)} \\
		&\qquad+g(x){T(h_1,\dots,h_m)(x})\\
		&=:{\rm I_1}(x)+{\rm I_2}(x).
	\end{split}
    \end{equation*}
Therefore, ${\rm I_1}(x)$ is supported on $B(x_0,r)$ and ${\rm I_2}(x)$ is supported on $B(y_0,r)$.

 We first estimate ${\rm I_1}(x)$. For $x\in B(x_0,r)$,
	we have
   \begin{align*}
   	  |{\rm I_1}|
   	&\lesssim\frac{\|a\|_{L^{\infty}}}{(Nr)^{-mn}g(B_{l})\prod_{j\ne l}\nu_j(B_j)}\\
   &\qquad \times\int_{\prod_{j=1}^{m}B(y_j,r)}\frac{|x-x_0|^{\gamma}g(z_l)\prod_{j\ne l}h_j(z_j)}{(\sum_{i=1}^{m}|z_l-z_i|+|z_l-x_0|)^{mn+\gamma}}dz_1\cdots dz_m\\
   	&\lesssim\frac{r^{-n}}{(Nr)^{-mn}g(B_{l})\prod_{j\ne l}\nu_j(B_j)}
   \frac{r^{\gamma}g(B_{l})\prod_{j\ne l}\nu_j(B_j)}{(Nr)^{mn+\gamma}}\\
   	&\lesssim \frac{1}{N^{\gamma} r^n}.
   \end{align*}
   Hence we obtain that
   \begin{equation}\label{I1}
    \begin{aligned}
    |{\rm I_1}(x)\lesssim \frac{1}{N^{\gamma} r^n}\chi_{B(x_0,r)}(x).
            \end{aligned}
    \end{equation}

For the term ${\rm I_2}(x)$, it follows from the definitions of $g(x)$, $h_i(x)$ that
  	\begin{align*}
		&|T(h_1,\cdots,h_m)(x)|\\
		&=\frac{1}{ |T^*_l(h_1,\dots,h_{l-1},g,h_{l+1},\dots,h_m)(x_0) |} \\
		&\qquad \times\Big|\int_{\prod_{j\ne l}B(y_j,r)\times B(x_0,r)} \big(K(z_1,\dots,z_{l-1},x_0,z_{l+1},\dots,z_m)(x_0)\\
		&\qquad\qquad-K(z_1,\dots,z_{l-1},x,z_{l+1},\dots,z_m)(x_0)\big)a(z_l)\prod_{j\neq l}h_{j}(z_{j})dz_1\cdots dz_m\Big|\\
		&\lesssim\frac{\|a\|_{L^{\infty}}}{(Nr)^{-mn}g(B_{l})\prod_{j\ne l}\nu_j(B_j)}\\
&\qquad \times \frac{g(B_{l})}{\nu_{l}(B_{l})}
	\int_{\prod_{j\ne l}B(y_j,r)\times B(x_0,r)}\frac{|x-x_0|^{\gamma}\nu_{l}(z_{l})\prod_{j\ne l}|h_j(z_j)|}{(\sum_{i=1}^{m}|z_l-z_i|+|z_l-x_0|)^{mn+\gamma}}dz_1\cdots dz_m\\
		&\lesssim \frac{1}{N^{\gamma} r^n},
    \end{align*}
where in the second equality we use the cancellation property of the atom $a(z_l)$. Thus
   \begin{equation}\label{I2}
    \begin{aligned}
    |{\rm I_2}(x)|\lesssim\frac{g(x)}{N^{\gamma} r^n}\chi_{B(y_l,r)}(x).
        \end{aligned}
    \end{equation}
	Combining the estimates \eqref{I1} and \eqref{I2}, we obtain that
\begin{equation}\label{size}
\begin{aligned}
	  | a(x)- {\textstyle \prod_{l}(g,h_1,\dots ,h_m)(x)} |\lesssim \frac{1}{N^{\gamma} r^n}\chi_{B(x_{0},r)}(x)+\frac{g(x)}{N^{\gamma} r^n}\chi_{B(x_{l},r)}(x)
	\end{aligned}
\end{equation}
with $\|g\|_{L^{p}(\mathbb{R}^n)}\leq |B_{l}|^{1/p}$. In addition, we point out that
\begin{equation}\label{can}
\begin{aligned}
	\int_{\mathbb{R}^n}  [a(x)- {\textstyle \prod_{l}(g,h_1,\dots ,h_m)(x)}  ]dx=0,
	\end{aligned}
\end{equation}
since the atom $a(x)$ has cancellation property and the second integral equals $0$ just by the definitions of ${\textstyle \prod_{l}}$.
    Then the size estimate \eqref{size} and the cancellation \eqref{can}, together with Lemma \ref{lem-main11}, imply that
    $$   \| a(x)- {\textstyle \prod_{l}(g,h_1,\dots ,h_m)(x)}   \|_{H^1(\mathbb{R}^n)}\le C\frac{\log N}{N^{\gamma}}.$$
     For $N$ sufficiently large such that
     $$\frac{C\log N}{N^{\gamma}}<\epsilon,$$
the result follows from here.
\end{proof}

\begin{lemma} \label{lem-main4}
Let $1\le l\le m$, $\frac{n}{n+\gamma}<\rho<1$, $1< p_{1},\dots,p_{m},q<\infty$ and $\omega\in A_{\vec{P},q}$. Then for any $\epsilon>0$, there exist $N>0$ and $C>0$ such that for any $H^\rho(\mathbb{R}^n)$-atom $a(x)$, there exist $g\in L^{q}(\mu_{\vec{\omega}}^{1-q'})$, $h_i\in L^{p_i}(\omega_i^{p_i})$, $i=1,2,\cdots,m$, with $\sum_{i=1}^{m}\frac{1}{p_i}-\frac{1}{q}=\frac{1}{\rho}-1$, such that
$$\|a- {\textstyle\prod_{l}}(g,h_1,\cdots,h_m)\|_{H^\rho(\mathbb{R}^n)}<\varepsilon,$$
and $$\|g\|_{L^q(\mu_{\vec{\omega}}^{1-q'})} \|h_1\|_{L^{p_1}(\omega_1^{p_1})}\cdots\|h_m\|_{L^{p_m}(\omega_m^{p_m})}\le CN^{mn(1+L)}.$$
\end{lemma}
\begin{proof}
	Let $a(x)$ be an $H^{\rho}(\mathbb{R}^n)$-atom, supported in $B(x_0,r)$, satisfying that
	$$ \int_{\mathbb{R}^n}a(x)dx=0\qquad and \qquad   \|a\|_{L^{\infty}(\mathbb{R}^n)}\le |B(x_{0},r)|^{-1/\rho}.$$

Similar to the lemma \ref{lem-main3}, we can choose $y_1,\dots,y_m$ such that $|y_{i}-y_{i+1}|=Nr$ with $i=1,\cdots,m-1$.
	We write $B_i=B(y_i,r)$, $\mu_i=\omega_i^{-p_i^{\prime}}\in A_{mp'_{i}}\subset A_{\infty}$ and set
	\begin{equation*}
		\begin{split}
			&g(x):=\Big(\frac{|B_{l}|\mu_{\vec{\omega}}(x)}{\mu_{\vec{\omega}}(B_{l})}\Big)^{1/q}\chi_{B_l}(x),\\
			&g(B_{l}):=\int_{B(y_l,r)}g(z_{l})dz_{l},\\
			&h_j(x):=\mu_{j}(x)\chi_{B_j}(x),\qquad j\ne l,\\
			& h_l(x)=\frac{a(x)}{T_l(h_1,\dots,h_{l-1},g,h_{l+1},\dots,h_m)(x_0)}\frac{g(B_{l})\mu_{l}(x)}{\mu_{l}(B_{l})}\chi_{B_l}(x).
		\end{split}
	\end{equation*}
	
	For the specific choice of the functions
	$h_1,\dots ,h_{l-1},g,h_{l+1},\dots,h_m$ as above, there exists a positive constant $C$ such that
	\begin{equation}\label{lem4-1}
    \begin{aligned}
    	&|T_{l}(h_1,\dots,h_{l-1},g,h_{l+1},\dots,h_m)(x_0)|\\
    	&\geq C\int_{B_1\times\cdots\times B_m}\frac{g(z_{l})\prod_{j\ne l}\mu_{j}(z_j)}{(|x_0-z_1|+\cdots+|x_0-z_m|)^{mn}}dz_1\cdots dz_m\\
    	&\ge C(Nr)^{-mn}g(B_{l})\prod_{j\ne l}\mu_j(B_j).
    \end{aligned}
    \end{equation}
	From the definitions of the functions $g(x)$, we obtain that $supp\ g=B(y_l,r)$. Moreover,
\begin{equation}\label{lem4-2}
    \begin{aligned}	\|g\|_{L^{q^{}\prime}(\mu_{\vec{\omega}}^{1-q^{\prime}})}=\Big(\frac{|B_l|}{\mu_{\vec{\omega}}(B_l)}\Big)^{1/q}|B_l|^{1/q^{\prime}}=\frac{|B_l|}{\mu_{\vec{\omega}}(B_l)^{1/q}}.
    \end{aligned}
    \end{equation}
By \eqref{lem4-1}, we have
	\begin{equation}\label{lem4-3}
    \begin{aligned}
		\|h_j\|_{L^{p_j}(\omega_j^{p_j})}\lesssim \mu_j(B_j)^{\frac{1}{p_j}}
    \end{aligned}
    \end{equation}
	for $j\neq l$. Also, for $j=l$,
	\begin{equation}\label{lem4-4}
    \begin{aligned}
		 \|h_{l} \|_{L^{p_{l}}(\omega_l^{p_l})}
		& = \frac{1}{ |T_{l} (h_{1}, \ldots, h_{l-1}, g, h_{l+1}, \ldots, h_{m} ) (x_{0} ) |}\frac{g(B_{l})}{\mu_{l}(B_{l})}\|a\|_{L^{p_{l}}(\mu_{l})} \\
		&\lesssim r^{-n/s}(Nr)^{mn}\mu_{l}(B_{l})^{-1/p'_{l}}\prod_{j\neq l}\mu_j(B_j)^{-1}.
    \end{aligned}
    \end{equation}
Combining the estimates \eqref{lem4-2}, \eqref{lem4-3} and \eqref{lem4-4}, we arrive at
	\begin{align*}
		\|g\|_{L^{q^{\prime}}(\mu^{1-q^{\prime}}_{\vec{\omega}})} \|h_{1} \|_{L^{p_{1}}(\omega_1^{p_1})}\cdots \|h_{m} \|_{L^{p_{m}(\omega_m^{p_m})}}
		&\lesssim r^{n/s'}(Nr)^{mn}\mu_{\vec{\omega}}(B_l)^{-1/q}\prod_{j=1}^{m}\mu_j(B_j)^{-1/p_j^{\prime}}.
	\end{align*}
	From the fact that $\sum_{j=1}^{m}\frac{1}{p_{j}}-\frac{1}{q}=\frac{1}{\rho}-1$ and
	$$1\leq   \Big(\frac{1}{|B|}\int_{B}\prod_{i=1}^{m}\omega_i(x)^{q}dx \Big)^{\frac{1}{q}} \prod_{i=1}^{m}  \Big(\frac{1}{|B|}\int_{B}\omega_{i}(x)^{-p_i^\prime} dx \Big)^{\frac{1}{p_i^{\prime}}}.$$
The same argument as \eqref{wzs}, we obtain
	$$\|g\|_{L^{q^{\prime}}(\mu^{1-q^{\prime}}_{\vec{\omega}})} \|h_{1} \|_{L^{p_{1}}(\omega_1^{p_1})}\cdots \|h_{m} \|_{L^{p_{m}(\omega_m^{p_m})}}\lesssim N^{mn(1+L)}.$$
	Next, we have
	\begin{equation*}
		\begin{split}
			&a(x)- {\textstyle \prod_{l}}(g,h_1,h_2,\dots ,h_m)(x)\\
			&=a(x)-  [h_lT^*_l(h_1,\dots,h_{l-1},g,h_{l+1},\dots,h_m)-gT_l(h_1,\dots,h_m)(x)  ] \\
			&=a(x)\frac{T_l^*(h_1,\dots,h_{l-1},g,h_{l+1},\dots,h_m)(x_0)-T_l^*(h_1,\dots,h_{l-1},g,h_{l+1},\dots,h_m)(x)}{T_l^*(h_1,\dots,h_{l-1},g,h_{l+1},\dots,h_m)(x_0)} \\
			&\quad+g(x){T(h_1,\dots,h_m)(x})\\
			&=:{\rm II_1}(x)+{\rm II_2}(x)\\
		\end{split}
	\end{equation*}
	It is obvious that ${\rm II_1}(x)$ is supported on $B(x_0,r)$ and ${\rm II_2}(x)$ is supported on $B(y_0,r)$.
	
	We first estimate ${\rm II_1}(x)$. For $x\in B(x_0,r)$,
	we have
   \begin{align*}
   	  | {\rm II_1}(x)   |
   	&\lesssim\frac{\|a\|_{L^{\infty}}}{(Nr)^{-mn}g(B_{l})\prod_{j\ne l}\nu_j(B_j)}\\
   &\qquad \times\int_{\prod_{j=1}^{m}B(y_j,r)}\frac{|x-x_0|^{\gamma}g(z_l)\prod_{j\ne l}h_j(z_j)}{(\sum_{i=1}^{m}|z_l-z_i|+|z_l-x_0|)^{mn+\gamma}}dz_1\cdots dz_m\\
   	&\lesssim\frac{r^{-n/\rho}}{(Nr)^{-mn}g(B_{l})\prod_{j\ne l}\mu_j(B_j)}
   \frac{r^{\gamma}g(B_{l})\prod_{j\ne l}\mu_j(B_j)}{(Nr)^{mn+\gamma}}\\
   	&\lesssim \frac{1}{N^{\gamma} r^{n/\rho}}.
   \end{align*}
	Hence we obtain that
	$$|{\rm II_1}(x)\lesssim \frac{1}{N^{\gamma} r^{n/\rho}}\chi_{B(x_0,r)}(x).$$
Meanwhile, we have
	$$|{\rm II_2}(x)|\lesssim\frac{g(x)}{N^{\gamma} r^{n/\rho}}\chi_{B(y_l,r)}(x).$$
with $\|g\|_{L^{q}(\mathbb{R}^n)}\leq |B_{l}|^{1/q}$. Combining the estimates of ${\rm II_1}(x)$ and ${\rm II_2}(x)$, we obtain that
	\begin{align}\label{II1}
		  | a(x)- {\textstyle \prod_{l}(g,h_1,\dots ,h_m)(x)} |\lesssim \frac{1}{N^{\gamma} r^{n/\rho}}\chi_{B(x_{0},r)}(x)+\frac{g(x)}{N^{\gamma} r^{n/\rho}}\chi_{B(x_{l},r)}(x)
	\end{align}
and
	\begin{align}\label{II2}
		\int_{\mathbb{R}^n}  [a(x)- {\textstyle \prod_{l}(g,h_1,\dots ,h_m)(x)} ]dx=0,
   \end{align}
	Then the size estimate \eqref{II1} and the cancellation \eqref{II2}, together with Lemma \ref{lem-main1}, imply that
	$$   \| a(x)- {\textstyle \prod_{l}(g,h_1,\dots ,h_m)(x)}   \|_{H^\rho(\mathbb{R}^n)}\lesssim \frac{N^{n(1/\rho-1)}}{N^{\gamma}}.$$
	Due to $\frac{n}{n+\gamma}<\rho<1$, for $N$ sufficiently large such that
	$$\frac{CN^{n(1/\rho-1)}}{N^{\gamma}}<\epsilon,$$
then, the lemma is proved.
\end{proof}

\section{Proofs of Main Theorems \ref{main1}-\ref{main2}}\label{s4}

The aim of this note is to prove the following results which give the factorization theorems in terms of multilinear Calder\'{o}n-Zygmund operators on weighted Lebesgue spaces without individual condition on the weight functions. With this approximation results above, Theorem \ref{main2} is completely analogous to Theorem \ref{main1}, with a small
difference, so we only prove the main Theorem \ref{main1}.

 \vspace{0.3cm}

 \noindent
 \textbf{Proof of Theorem \ref{main1}.}
 Utilizing the atomic decomposition, for any $f\in H^1(\mathbb{R}^n)$ we can find a sequence $ \{\lambda_s^1 \}\in \ell^1$ and sequence of $H^1(\mathbb{R}^n)$-atom
    $ \{a_{s}^{1}\}$ so that
    $$f=\sum_{s=1}^{\infty}\lambda _s^1a_s^1 \qquad  and \qquad \sum_{s=1}^{\infty}|\lambda_s^1|\le C  \|f\|_{H^1(\mathbb{R}^n)}.$$
    We explicitly track the implied absolute constant $C$ appearing from the atomic decomposition since it will play a role in the convergence of the algorithm.

Fix $\varepsilon >0$ small enough such that $C\varepsilon <1$. We apply Lemma \ref{lem-main3} to each atom $a_s^1$, there exists $g_s^1\in L^{p^{\prime}}(\nu^{1-p^{\prime}}_{\vec{\omega}}),h_{s,1}^1\in L^{p_1}(\omega_1),\cdots ,h_{s,m}^1\in L^{p_m}(\omega_m)$ with
    $$  \| a_s^1- {\textstyle \prod_{l}}(g_s^1,h_{s,1}^1,\cdot,h_{s,m}^1) \|_{H^1({\mathbb{R}^n)}} <\varepsilon ,\qquad\forall s$$ and
    $$  \| g_s^1   \|_{L^{p^\prime}(\nu^{1-p^{\prime}}_{\vec{\omega}})}   \| h_1   \|_{L^{p_1}(\omega_1)}\cdots   \| h_m   \|_{L^{p_m}(\omega_m)}\le CM^{mn(1+L)}.$$
Note that
   	\begin{equation*}
    \begin{split}
   	& f=\sum_{s=1}^{\infty}\lambda _s^1a_s^1=: M_1+E_1,
   	\end{split}
    \end{equation*}
    where
       	\begin{equation*}
    \begin{split}
   	& M_1=\sum_{s=1}^{\infty}\lambda_s^1{\textstyle\prod_{l}}(g_s^1,h_{s,1}^1,\dots,h_{s,m}^1)(x)\\
  	& E_1=\sum_{s=1}^{\infty}\lambda_s^1(a_s^1-{\textstyle\prod_{l}}(g_s^1,h_{s,1}^1,\dots,h_{s,m}^1)).
   	\end{split}
    \end{equation*}
    Observe that
       	\begin{equation*}
    \begin{split}
    \|E_1\|_{H^1(\mathbb{R}^n)}&\le{\textstyle \sum_{s=1}^{\infty}}|\lambda_s^1|\ \|a_s^1-{\textstyle\prod_{l}(g_s^1,h_{s,1}^1,\dots,h_{s,m}^1)}\|_{H^1(\mathbb{R}^n)}\\
    &\le \varepsilon {\textstyle \sum_{s=1}^{\infty}}|\lambda_s^1|\le \varepsilon C\|f\|_{H^1(\mathbb{R}^n)}.
       	\end{split}
    \end{equation*}

    Now,we iterate the construction on the function $E_1$.  Since $E_1\in H^1(\mathbb{R}^n)$, we can apply the atomic decomposition in $H^1(\mathbb{R}^n)$ to find a sequence $ \{\lambda_s^2 \}\in l^1$ and sequence of $H^1(\mathbb{R}^n)$-atom $  \{a_{s}^{2}   \}$ so that $E_1=\sum_{s=1}^{\infty}\lambda _s^2a_s^2$. and $$\sum_{s=1}^{\infty}|\lambda_s^2|\le C  \| E_1  \|_{H^1(\mathbb{R}^n)} \le \varepsilon C^2  \| f  \|_{H^1(\mathbb{R}^n)}.$$

    Again, applying Lemma \ref{lem-main3} to each atom $a_s^2$, there exists $g_s^2\in L^{p^{\prime}}(\nu^{1-p^{\prime}}_{\vec{\omega}}),h_{s,1}^2\in L^{p_1}(\omega_1),\cdots ,h_{s,m}^2\in L^{p_m}(\omega_m)$ with
    $$  \| a_s^2- {\textstyle \prod_{j,l}}(g_s^2,h_{s,1}^2,\cdot,h_{s,m}^2) \|_{H^1({\mathbb{R}^n)}} <\varepsilon ,\qquad\forall s$$
    We then have that:
   	\begin{equation*}
   	\begin{split}
    & E_1=\sum_{s=1}^{\infty}\lambda _s^2a_s^2={\textstyle \sum_{s=1}^{\infty}}\lambda_s^2{\textstyle\prod_{l}(g_s^2,h_{s,1}^2,\dots,h_{s,m}^2)(x)}+{\textstyle \sum_{s=1}^{\infty}}\lambda_s^2(a_s^2-{\textstyle\prod_{l}(g_s^2,h_{s,1}^2,\dots,h_{s,m}^2)})\\
   	&=: M_2+E_2.
   	\end{split}
    \end{equation*}
	As before, observe that
	\begin{equation*}
	\begin{split}
	\|E_2\|_{H^1(\mathbb{R}^n)} &\le{\textstyle \sum_{s=1}^{\infty}}|\lambda_s^2|\ \|a_s^2-{\textstyle\prod_{l}(g_s^2,h_{s,1}^2,\dots,h_{s,m}^2)}\|_{H^1(\mathbb{R}^n)}\\
    &\le \varepsilon{\textstyle\sum_{s=1}^{\infty}}|\lambda_s^2|
	\le (\varepsilon C)^2\|f\|_{H^1(\mathbb{R}^n)}.
	\end{split}
	\end{equation*}
    And, this implies for $f$ that we have:
   	\begin{equation*}
    \begin{split}
   	& f=\sum_{s=1}^{\infty}\lambda _s^1a_s^1={\textstyle \sum_{s=1}^{\infty}}\lambda_s^1{\textstyle\prod_{l}(g_s^1,h_{s,1}^1,\dots,h_{s,m}^1)(x)}
    		+{\textstyle \sum_{s=1}^{\infty}}\lambda_s^1(a_s^1-{\textstyle\prod_{l}(g_s^1,h_{s,1}^1,\dots,h_{s,m}^1)})\\
   	&= M_1+E_1= M_1+M_2+E_2\\
    & =\sum_{k=1}^{2}\sum_{s=1}^{\infty }\lambda_s^k{\textstyle\prod_{l}(g_s^k,h_{s,1}^k,\dots,h_{s,m}^k)}+E_2.
   	\end{split}
    \end{equation*}

Repeating this construction for each $1\le k\le K$ produces functions $g_s^k\in L^{p^{\prime}}(\nu^{1-p^{\prime}}_{\vec{\omega}}),h_{s,1}^k\in L^{p_1}(\omega_1),\cdots ,h_{s,m}^k\in L^{p_m}(\omega_m)$ with
 $$  \| g_s^k   \|_{L^{p^\prime}(\nu^{1-p^{\prime}}_{\vec{\omega}})}   \| h_{s,1}^k   \|_{L^{p_1}(\omega_1)}\cdots   \| h_{s,m}^k   \|_{L^{p_m}(\omega_m)}\le C(\varepsilon, L), \quad \forall s,$$
  sequences $ \{\lambda_s^k \}\in l^1$ with $ |\lambda_s^k|_{l_1}\le \varepsilon^{k-1} C^k  \| f  \|_{H^1(\mathbb{R}^n)}$, and a function $E_K\in H^1(\mathbb{R}^n)$ with $$\|E_K\|_{H^1(\mathbb{R}^n)}\le (C\varepsilon)^K  \| f  \|_{H^1(\mathbb{R}^n)}.$$
	So that
	$$f=\sum_{k=1}^{K}\sum_{s=1}^{\infty }\lambda_s^k{\textstyle\prod_{l}(g_s^k,h_{s,1}^k,\dots,h_{s,m}^k)}+E_k.$$
	Passing $K\to \infty $ gives the desired decomposition of $$f=\sum_{k=1}^{\infty}\sum_{s=1}^{\infty}\lambda_s^k{\textstyle\prod_{l}(g_s^k,h_{s,1}^k,\dots,h_{s,m}^k)}$$
with
$$\sum_{k=1}^{\infty}\sum_{s=1}^{\infty}|\lambda_s^k|\le \sum_{k=1}^{\infty}\varepsilon ^{-1}(C\varepsilon)^K\|f\|_{H^1(\mathbb{R}^n)}=\frac{C}{1-\varepsilon C}\|f\|_{H^1(\mathbb{R}^n)}.$$
By Lemma \ref{lem-main2}, we have that
 $$  \|{\textstyle \prod_{l}(g^k_{s,1},h^{k}_{s,1},\dots ,h^k_{s,m})}\|_{H^1(\mathbb{R}^n)}\lesssim  \| g_s^{k} \|_{L^{p^{\prime}}(\nu^{1-p^{\prime}}_{\vec{\omega}})}  \|h_{s,1}^k  \|_{L^{p_1}(\omega_1)} \cdots   \|h_{s,m}^k  \|_{L^{p_m}(\omega_m)}.$$
    It is immediate that for any representation of $f$ as in \eqref{f=1}, the norm $ \| f\|_{H^1(\mathbb{R}^n)}$ is bounded by
    $$ C\inf  \{ \sum_{k=1}^{\infty}\sum_{s=1}^{\infty}\lambda_s^k \| g_s^{k} \|_{L^{p^{\prime}}(\nu^{1-p^{\prime}}_{\vec{\omega}})}  \|h_{s,1}^k  \|_{L^{p_1}(\omega_1)} \cdots   \|h_{s,m}^k  \|_{L^{p_m}(\omega_m)}\ :\ f\ \text{satisfies} \ \eqref{f=1}  \} .$$
Thus, we have completed the proof of Theorem \ref{main1}. \qed

\section{Applications}

As a direct application, we will give the characterization of $BMO(\mathbb{R}^n)$ and Lipschitz spaces via commutators of the multilinear Calder\'{o}n-Zygmund operator on weighted Lebesgue spaces, without individual conditions on the weights class.

It has been an open question whether commutator of multilinear Calder\'{o}n-Zygmund operators can be used to characterize $BMO$? Chaffee \cite{Chaf2016} first considered this problem and obtained the following characterized theorem.

\medskip

{\bf Theorem A} (cf. \cite{Chaf2016})\quad Suppose that $K$ is a homogeneous function of degree $-mn$, and there exists a ball $\mathbb{B}\subset \mathbb{R}^{mn}$ such that $1/K$ can be expended to a Fourier series in $\mathbb{B}$. If $1<p_1,\cdots,p_m,{p}<\infty$, $\frac 1p=\frac 1{p_1}+\cdots+\frac 1{p_m}$,
$$b\in BMO\,\Longleftrightarrow\, [b,T]_i:\, L^{p_1}(\mathbb{R}^n)\times\cdots\times L^{p_m}(\mathbb{R}^n)\to L^p(\mathbb{R}^n).$$

Later, the necessity of bounded multilinear commutator was extended by \cite{ChafCruz2018}(or \cite{GLW2020} for quasi-Banach space) in a very general structure. It is should be pointed that $\vec{\omega}\in A_{\vec{P}}$ does not imply $\omega_{k}\in L^{1}_{loc}$ for any $k$ (\cite[Remark 7.2]{LOPTT2009}), which showed that $\omega_{k}\in A_{p_{k}}$ may be not valid. Thus, they proved the necessity theory for multilinear commutators in the weighted case with making stronger assumption.

\medskip

{\bf Theorem B} (cf. \cite{ChafCruz2018}).\quad{\it Given $\vec{P}$ with $p>1$, suppose $\omega_{k}\in A_{p_{k}}, k=1,\cdots,m,$ and $\vec{\omega}\in A_{\vec{P}}$. If $T$ is a regular bilinear singular integral, and $b$ is function such that for $i=1,\cdots,m$, $[b,T]_{i}: L^{p_{1}}\times\cdots\times L^{p_{m}}(\omega_{m})\rightarrow L^{p}(\nu_{\vec{\omega}})$, then $b\in BMO(\mathbb{R}^n)$.}

\medskip

Recently, the necessity of bounded multilinear commutator with the genuinely multilinear weights had been resolved in \cite{L2020} using sparse domination and in \cite{Wang} by Cauchy integral trick. In this paper, we revisit this problem using different techniques. In contrast to previous results, one advantage
of the approach taken in this paper is that it provides for a constructive algorithm to produce the weak factorization of Hardy spaces.

\begin{theorem} \label{a1}
Let $1\le l\le m$,  $1< p_{1},\dots,p_{m}<\infty$, $\frac{1}{p_1}+\dots+\frac{1}{p_m} =\frac{1}{p}$ and $\vec{\omega}\in A_{\vec{P}}$. Then the commutator $  [b,T  ]_{l}(f_1,\cdots,f_m)(x)$ is a bounded map form $L^{p_1}(\omega_1)\times \dots \times L^{p_m}(\omega_m)$ to $L^{p}(\nu_{\vec{\omega}})$ if and only if $b \in BMO(\mathbb{R}^n)$.
\end{theorem}
\begin{proof}
The upper bound in this theorem is contained in \cite{LOPTT2009}. For the lower bound, suppose that $f\in{H^1(\mathbb{R}^n)}$, using the weak factorization in Theorem \ref{main1}  we obtain
	\begin{equation*}
	\begin{split}
 \langle b,f   \rangle_{L^2(\mathbb{R}^n)}&=\sum_{k=1}^{\infty }\sum_{s=1}^{\infty }\lambda_s^k   \langle b,{\textstyle\prod_{l}(g_s^k,h_{s,1}^k,\dots,h_{s,m}^k)}   \rangle_{L^2(\mathbb{R}^n)}\\
	&=\sum_{k=1}^{\infty }\sum_{s=1}^{\infty }\lambda_s^k  \langle g_s^k,{ [ b,T  ]_l(h_{s,1}^k,\dots,h_{s,m}^k)}   \rangle_{L^2(\mathbb{R}^n)}.
	\end{split}
	\end{equation*}
    Hence, by the weighted boundedness of $[b,T]_l$, we arrive at
  	\begin{equation*}
  	\begin{split}  
  |\langle b,f\rangle_{L^2(\mathbb{R}^n)}|&\le \sum_{k=1}^{\infty }\sum_{s=1}^{\infty }|\lambda_s^k|\ \|g_s^k\|_{L^{p^{\prime}}(\nu^{1-p^{\prime}}_{\vec{\omega}})}
  		\|{ [ b,T ]_l(h_{s,1}^k,\dots,h_{s,m}^k)}\|_{L^p(\nu_{\vec{\omega}})}\\
  	&\lesssim \sum_{k=1}^{\infty} \sum_{s=1}^{\infty} |\lambda_{s}^{k} | \|g_{s}^{k} \|_{L^{p^{\prime}} (\nu^{1-p^{\prime}}_{\vec{\omega}} )} \prod_{i=1}^{m} \|h_{s,i}^{k} \|_{L^{p_{i}} (\omega_i)}\\
  	& \lesssim \|g_{s}^{k} \|_{L^{p^{\prime}} (\nu^{1-p^{\prime}}_{\vec{\omega}} )} \prod_{i=1}^{m} \|h_{s,i}^{k} \|_{L^{p_{i}} (\omega_i)}. 
   	\end{split}
   \end{equation*}
   From the duality theorem between $H^1(\mathbb{R}^n)$ and $BMO(\mathbb{R}^n)$, it follows that $b\in BMO(\mathbb{R}^n)$.
   \end{proof}

Note that in \cite{FS1972}, the authors show that the dual of $H^\rho(\mathbb{R}^n)$ and $Lip_{\alpha}(\mathbb{R}^n)$, where $\alpha=n(1/\rho-1)$. In this paper, we give a new characterization of  $Lip_{\alpha}(\mathbb{R}^n)$ space for $mn$-homogeneous multilinear Caldr\'{o}n-Zygmund operators, using the duality theorem between $H^\rho(\mathbb{R}^n)$ and $Lip_{\alpha}(\mathbb{R}^n)$.

\begin{theorem}\label{a2}
Let $b\in L_{loc}(\mathbb{R}^n)$, $0<\alpha<n$, $\frac{1}{p_1}+\cdots+\frac{1}{p_m}-\frac{1}{q}=\frac{\alpha}{n}$ and $\vec{\omega}\in A_{\vec{P},q}$. Then
	\begin{align*}
		 \|b \|_{Lip_{\alpha}(\mathbb{R}^n)}\approx \|[b,T]_l \|_{L^{p_1}(\omega_1^{p_1})\times\cdots\times L^{p_m}(\omega_m^{p_m}) \rightarrow L^q(\mu_{\vec{\omega}})}.
	\end{align*}
\end{theorem}
\begin{proof}
Let $b\in Lip_{\alpha}(\mathbb{R}^n)$. The weighted boundedness of $[b,T]_l$ follows from the fact that
$$|[b,T]_l(f_{1},\cdots,f_{m})(x)|\leq \|b\|_{Lip_{\alpha}(\mathbb{R}^n)}I_{\alpha,m}(|f_{1},\cdots,f_{m}|)(x).$$

Now we consider the opposite case. Suppose $f\in H^\rho(\mathbb{R}^n)$, then by Theorem \ref{main2}, there exist sequences $  \{ \lambda_{s}^{k} \}\in \ell_{s}$ and functions $ g_{s}^{k}\in L^{q'}(\mu_{\vec{\omega}}^{1-q'}),h_{s,1}^{k}\in L^{p_{1}}(\omega_1^{p_1}),\cdots ,h_{s,m}^{k}\in L^{p_{m}}(\omega_{m}^{p_m}),$ such that
\begin{align*}
	f=\sum_{k=1}^{\infty}\sum_{s=1}^{\infty}\lambda_{s}^{k}
	{\textstyle \prod_{l}}(g_{s}^{k},h_{s,1}^{k},\dots ,h_{s,m}^{k}) \quad in \quad  H^\rho(\mathbb{R}^n)
\end{align*}
and so,
\begin{align*}
 \langle b,f \rangle_{L^2(\mathbb{R}^n)}
&= \Big\langle b,\sum_{k=1}^{\infty}\sum_{s=1}^{\infty}\lambda_{s}^{k}	{\textstyle\prod_{l}}(g_{s}^{k},h_{s,1}^{k},\cdots,h_{s,m}^{k}) \Big\rangle_{L^2(\mathbb{R}^n)}\\
&=\sum_{k=1}^{\infty}\sum_{s=1}^{\infty}\lambda_s^k  \Big\langle b, {\textstyle\prod_{l}}(g_s^k,h_{s,1}^k,\cdots,h_{s,m}^k) \Big\rangle_{L^2(\mathbb{R}^n)}\\
&=\sum_{k=1}^{\infty}\sum_{s=1}^{\infty}\lambda_s^k  \Big\langle g_s^k, [b,T]_l(h_{s,1}^k,\cdots,h_{s,m}^k) \Big\rangle_{L^2(\mathbb{R}^n)}.
\end{align*}
This implies, by Holder's inequality and $[b,T]_l$ maps  $L^p_{1}(\omega_1^{p_1})\times \cdots \times L^{p_{m}}(\omega_m^{p_m})$ to $L^{q}(\mu_{\vec{\omega}}^{1-q'})$,
\begin{align*}
| \langle b,f \rangle_{L^{2}(\mathbb{R}^n)}|&\le\sum_{k=1}^{\infty}\sum_{s=1}^{\infty}|\lambda_s^k|||g_s^k||_{L^{q'}(\mu_{\vec{\omega}}^{1-q'})}||[b,T]_l(h_{s,1}^k,\cdots,h_{s,m}^k)||_{L^{q}(\mu_{\vec{\omega}}^{1-q'})}\\
&\lesssim \sum_{k=1}^{\infty}\sum_{s=1}^{\infty}|\lambda_s^k|\|g_s^k\|_{L^{q'}(\mu_{\vec{\omega}}^{1-q'})}\prod_{i=1}^{m} \|h_{s,j}^k\|_{L^{p_i}(\omega_j^{p_i})}.
\end{align*}
By the duality between $H^\rho(\mathbb{R}^n)$ and $Lip_{\alpha}(\mathbb{R}^n)$, which shows that $b\in Lip_{\alpha}(\mathbb{R}^{n})$ and the result follows from here.
\end{proof}


\end{document}